\documentclass{article}
\usepackage{amsmath, amssymb, amsthm}
\usepackage{url}
\usepackage{hyperref}
\usepackage{authblk}

\newtheorem{theorem}{Theorem}     
\newtheorem{lemma}{Lemma}   
\newtheorem{proposition}{Proposition}   
\newtheorem{corollary}{Corollary}   
\theoremstyle{definition}
\newtheorem{remark}{Remark}   

\begin{document}

\title { \textbf{Sharp Bounds on Diminished Sombor Index}}

\author[1]{Meysam Taheri-Dehkordi\thanks{Corresponding author.}}
\author[2]{Amir Hossein Nokhodkar}
\author[2]{Gholam Hossein Fath-Tabar}

\affil[1]{University of Applied Science and Technology, Tehran, Iran} 
\vspace{0.3cm}
\affil[2]{Department of Pure Mathematics, Faculty of Mathematical Sciences,\\ University of Kashan, Kashan 87317- 53153, I. R. Iran}

\date{}
\maketitle

\vspace{-0.5cm}
\begin{center}
\footnotesize
\textit{Email:} {m.taheri@uast.ac.ir}, {a.nokhodkar@kashanu.ac.ir}, {fathtabar@kashanu.ac.ir}
\end{center}
\vspace{0.5cm}

\begin{abstract}
The Diminished Sombor index $(DSO)$ of a simple graph $G$ is a recently introduced degree-based topological index, defined as 
\[ 
DSO(G)=
\sum_{uv\in E(G)}
\frac{\sqrt{d_u^2+d_v^2}}
     {d_u+d_v}
\]
In this paper, we establish a collection of new sharp bounds for the $DSO$ index in terms of several fundamental graph parameters and well-known topological indices. In all cases, the proposed bounds rigorously specify the equality conditions, thereby providing a complete characterization of the extremal graphs. These findings deepen the theoretical understanding of the $DSO$ index and provide powerful analytical tools for interdisciplinary applications in mathematical chemistry.
\end{abstract}


\section{Introduction}
Suppose $G=(V(G),E(G))$ is a simple, connected graph, where $|V(G)|=n$ and $|E(G)|=m$ represent the order and the size of $G$. The degree of each vertex $u$ is denoted by $d_u$, while the maximum and minimum degrees of the graph are denoted by $\Delta$ and $\delta$, respectively. For additional graph-theoretic terminology and notation not explicitly defined here, we refer the reader to \cite{7, 8, 9}. In graph theory, a graph invariant is defined as a numerical quantity that remains unchanged under graph isomorphism. Within the domain of chemical graph theory, such invariants are conventionally referred to as topological indices. Graph invariants find extensive application in chemistry as molecular descriptors. Within the framework of chemical graph theory, numerous degree-based graph invariants have been proposed and subsequently employed in various studies. Among these, the First Zagreb index $M_1$ and the Second Zagreb index $M_2$ are two of the oldest and most well-known indices \cite{10, 11}. For a given graph $G$, these indices are formally defined as follows: 
\[
\displaystyle M_1(G)=\sum_{v\in V(G)} d_v^2 = \sum_{uv\in E(G)}(d_u+d_v),\qquad \displaystyle M_2(G)=\sum_{uv\in E(G)}d_u d_v.
\]
In the following, we present the definitions of the topological indices that are of particular interest in this paper.
The Randić index was proposed by Milan Randić in 1975 \cite{18}, and is defined as
\[
\displaystyle R(G)=\sum_{uv\in E(G)} \frac{1}{\sqrt{d_u d_v}}.
\]
The Geometric-arithmetic index is defined in \cite{19} as
\[
\displaystyle GA(G)=\sum_{uv\in E(G)}\frac{2\sqrt{d_ud_v}}{d_u+d_v}.
\]
The Albertson index, also known as the third Zagreb index \cite{21}, was introduced in \cite{20} and is defined in the following manner:
\[
\displaystyle Alb(G)=\sum_{uv\in E(G)}|d_u-d_v|.
\]
In addition to the Albertson index, another natural irregularity measure, sigma irregularity index, was proposed in \cite{29} and is defined below:
\[
\displaystyle \sigma(G)=\sum_{uv\in E(G)} (d_u-d_v)^2.
\]
The Harmonic index is defined in \cite{22} as
\[
\displaystyle H(G)=\sum_{uv\in E(G)}\frac{2}{d_u+d_v}.
\]
In \cite{23} the inverse sum index is defined as
\[
\displaystyle ISI(G)=\sum_{uv\in E(G)}\frac{d_u d_v}{d_u+d_v}.
\]
The forgotten index is defined in \cite{24} as
\[
\displaystyle F(G)=\sum_{uv\in E(G)}(d_u^2+d_v^2).
\]
The total irregularity index was introduced by Abdo et al. \cite{25} and is defined as
\[
\displaystyle irr_t(G)=\sum_{\{u,v\}\subseteq V(G)} |d_u-d_v|.
\]
The index $HM(G)$, referred to in the paper as the Hyper-Zagreb index, was introduced by Shirdel et al \cite{26}, and defined as
\[
\displaystyle HM(G)=\sum_{uv\in E(G)}(d_u+d_v)^2.
\]
A new topological index, known as the Sombor index, was introduced by Gutman \cite{28}. The Sombor index of a graph $G$ is defined as
\[
\displaystyle SO(G)=\sum_{uv\in E(G)}\sqrt{d_u^2+d_v^2}.
\]
For further details on the Sombor index and its variants, the reader is referred to \cite{14, 15, 16, 17}. The significant progress and expansion of research on the Sombor index have led to the creation of a growing family of related indices and new variants of Sombor index.
 
  Recently, a modification of the Sombor index was introduced as follows \cite{13}:
 \[
 \displaystyle DSO(G)=\sum_{uv\in E(G)}\frac{\sqrt{d_u^2+d_v^2}}{d_u+d_v}.
 \]
 In reference \cite{3}, this index is referred to as the Diminished Sombor index. The authors established bounds for this index. In \cite{30}, the author identified the tricyclic graph with a prescribed number of vertices that maximizes the $DSO$ index, and proceeded to analyze its characteristic structural properties. In a study by Guo and Wang \cite{31}, the extremal value of this index among molecular trees possessing a perfect matching was established, and the corresponding extremal structures were fully characterized. 
 A series of bounds relating the $DSO$ index to several well-established degree-based topological indices were derived in \cite{5}. More recently, the behavior of the $DSO$ index has been examined for graphs of fixed order and cyclomatic number at least three \cite{33}. Although the work in \cite{5} compared the $DSO$ index with some other topological indices, in this paper we provide a far more comprehensive comparison and identify several improved sharp bounds. This is the main reason for the length of the present paper. In particular, wherever a bound in \cite{5} admitted a sharper form, we have established such a refinement, and we fully characterize the equality cases. These new results substantially advance the current understanding of the $DSO$ index and contribute to a deeper theoretical insight into this invariant.

\section{Comparison with Structural Parameters}
In this section, new bounds for the $DSO$ index are presented in terms of structural graph parameters.
All graphs considered in this paper are connected.
\begin{theorem}\label{Delta-delta}
Let $G$ be a graph with $m>0$ edges, and minimum degree $\delta\ge 1$. Then
\[
DSO(G) \le m\,\frac{\sqrt{\Delta^{2}+\delta^{2}}}{\Delta+\delta}.
\]
Equality holds if and only if $G$ is either regular  or bipartite semiregular with part degrees $\Delta$ and $\delta$.
\end{theorem}
\begin{proof}
Let $f(x)=\sqrt{1+x^2}/(1+x)$ for $x\ge 1$. Since
\[
f'(x)=\frac{x-1}{(1+x)^2\sqrt{1+x^2}}>0 \quad\text{for } x>1,
\]
$f$ is increasing on $[1,\infty)$. For every edge $uv\in E(G)$, assume without loss of generality $d_u\ge d_v$ and set $x=d_u/d_v$. Since $\delta\le d_v\le d_u\le \Delta$, we have $1\le x\le \Delta/\delta$, and
\[
\frac{\sqrt{d_u^2+d_v^2}}{d_u+d_v}=f(x)\le f(\Delta/\delta)=\frac{\sqrt{\Delta^2+\delta^2}}{\Delta+\delta}.
\]
Summing over all edges gives
\begin{equation}\label{eqn-step1}
DSO(G)\le m\cdot\frac{\sqrt{\Delta^2+\delta^2}}{\Delta+\delta},
\end{equation}
with equality if and only if $\{d_u,d_v\}=\{\Delta,\delta\}$ for every edge (i.e.\ $G$ is regular, or $G$ is bipartite semiregular with part degrees $\Delta$ and $\delta$).
\end{proof}
\begin{remark}
The upper bound of Theorem~\ref{Delta-delta} improves the bound
\[
DSO(G)\le \frac{\sqrt{2}}{2}m(\Delta-\delta+1).
\]
from Corollary 3.3 of \cite{5}. 
This follows from the following elementary inequality:
\[
\frac{\sqrt{\Delta^2+\delta^2}}{\Delta+\delta}
\le \frac{\sqrt{2}}{2}\,(\Delta-\delta+1),
\]
with equality if and only if $\Delta=\delta$.
Indeed, if $\Delta=\delta$, both sides equal $1/\sqrt{2}$. If $\Delta>\delta$, then $\Delta-\delta+1\ge 2$, so the right-hand side is at least $\sqrt{2}$, while the left-hand side is always strictly less than $1$ (since $\sqrt{\Delta^2+\delta^2}<\Delta+\delta$). Hence the inequality is strict in this case.
\end{remark}
Let $G$ be a  graph, and let $\lambda_1$ denote the largest eigenvalue of $G$.
It is well known that $m\le \frac{n\lambda_1}{2}$, with equality if and only if $G$ is regular (see \cite{4}). Hence, using Theorem \ref{Delta-delta}, one concludes that
\begin{corollary}
Let $G$ be a graph with largest eigenvalue $\lambda_1$. Then
\[
DSO(G)\le \frac{n\lambda_1}{2}\cdot\frac{\sqrt{\Delta^2+\delta^2}}{\Delta+\delta}.
\]
Equality holds if and only if $G$ is a regular graph.
\end{corollary}

\begin{theorem}\label{new:pendant}
Let $G$ be a graph with at least one non-pendant vertex.
Let $p$ be the number of pendant vertices and let $\delta_1$ be the minimum degree among non-pendant vertices.
Then
\[
DSO(G) \le p\,\frac{\sqrt{\Delta^2+1}}{\Delta+1}
+ (m-p)\,\frac{\sqrt{\Delta^2+\delta_1^2}}{\Delta+\delta_1}.
\]
Equality holds if and only if every edge incident to a pendant vertex has its other endpoint of degree $\Delta$, and every edge between two non-pendant vertices joins a $\Delta$-vertex to a $\delta_1$-vertex.
\end{theorem}
\begin{proof}
Separate the edge set into $E_1$ (edges incident to a pendant) and $E_2$ (edges between non-pendants). On $E_1$, one endpoint has degree $1$ and the other at most $\Delta$; by the monotonicity used in Theorem~\ref{Delta-delta}, each such edge contributes at most $\frac{\sqrt{\Delta^2+1}}{\Delta+1}$. On $E_2$, both degrees are between $\delta_1$ and $\Delta$, so each edge contributes at most $\frac{\sqrt{\Delta^2+\delta_1^2}}{\Delta+\delta_1}$. Summation yields the claim. 
Equality holds if and only if equality holds for every edge in both
families $E_1$ and $E_2$, which is equivalent to the stated degree
conditions.
\end{proof}
\begin{remark}
The hypothesis that $G$ has at least one non-pendant vertex guarantees that $\delta_1$ (the minimum degree among non-pendant vertices) is well defined and that $m-p$ (the number of edges joining two non-pendant vertices) is nonnegative.
The only connected graph violating this condition is $K_2$, for which the statement would not be meaningful (as $m-p=-1$ and $\delta_1$ does not exist).
Hence the theorem applies to all connected graphs of order at least $3$.
\end{remark}
\begin{remark}
Let
\[
f(x)=\frac{\sqrt{\Delta^2+x^2}}{\Delta+x},\qquad 1\le x\le \Delta.
\]
A direct differentiation gives
\[
f'(x)=\frac{\Delta(x-\Delta)}{\sqrt{\Delta^2+x^2}\,(\Delta+x)^2}\le 0,
\]
so $f$ is nonincreasing on $[1,\Delta]$ (strictly decreasing for $x<\Delta$).

Now assume that $G$ is connected, has pendant vertices, and possesses at least one non-pendant vertex (so $\delta=1$ and $\delta_1$ is well defined).
The bound of Theorem~\ref{Delta-delta} is $B_1 = m\,f(1)$.
The bound of Theorem~\ref{new:pendant} is
$
B_2 = p\,f(1)+(m-p)\,f(\delta_1),
$
where $p$ is the number of pendant vertices (hence also the number of pendant edges) and $\delta_1\ge 2$.
Because $f$ is decreasing, $f(\delta_1) < f(1)$.
Therefore
\[
B_2 = p f(1)+(m-p) f(\delta_1) \le p f(1)+(m-p) f(1) = m f(1) = B_1,
\]
with equality if and only if $m-p=0$, i.e., every edge is incident to a pendant vertex.
For a connected graph with a non-pendant vertex, this happens exactly when $G$ is a star $K_{1,m}$ (which forces $m\ge 2$).
Thus the bound of Theorem~\ref{new:pendant} is always at least as good as that of Theorem~\ref{Delta-delta} under the stated hypotheses, and it is strictly sharper whenever $G$ is not a star.
\end{remark}

\section{Comparison with Topological Indices}

In this section, we establish several new bounds for the $DSO$ index in terms of other well-known topological indices.

\begin{theorem}\label{DSO-M2-Delta}
Let $G$ be a  graph with $m>0$ edges. Then
\[
DSO(G)\le \sqrt{m^2-\frac{m}{2\Delta^2} M_2(G)}.
\]
Equality holds if and only if $G$ is a regular graph.
\end{theorem}

\begin{proof}
By the Cauchy–Schwarz inequality,
$
\left(\sum_{i=1}^m x_i\right)^2 \le m\sum_{i=1}^m x_i^2.
$
Set
$x_{uv}=\frac{\sqrt{d_u^2+d_v^2}}{d_u+d_v}.
$
Then
\begin{equation}\label{eqn12}
DSO(G)^2
\le 
m\sum_{uv\in E(G)} \frac{d_u^2+d_v^2}{(d_u+d_v)^2}.
\end{equation}
On the other hand,
$
d_u^2+d_v^2=(d_u+d_v)^2-2d_u d_v,
$
so
\begin{equation}\label{eqn13}
\frac{d_u^2+d_v^2}{(d_u+d_v)^2}
=1-\frac{2d_u d_v}{(d_u+d_v)^2}. 
\end{equation}
Since $d_u+d_v\le 2\Delta$ for every edge $uv$, we have $\frac{1}{(d_u+d_v)^2} \ge \frac{1}{4\Delta^2}$. Therefore,
\begin{equation}\label{eqn14}
\frac{2d_u d_v}{(d_u+d_v)^2}
\ge \frac{2d_u d_v}{4\Delta^2}
= \frac{d_u d_v}{2\Delta^2}. 
\end{equation}
From \eqref{eqn13} and \eqref{eqn14},
\[
\frac{d_u^2+d_v^2}{(d_u+d_v)^2}
\le 1-\frac{d_u d_v}{2\Delta^2}.
\]
Summing over all edges gives
\begin{equation}\label{eqn15}
\sum_{uv\in E(G)} \frac{d_u^2+d_v^2}{(d_u+d_v)^2}
\le 
\sum_{uv\in E(G)} \left(1-\frac{d_u d_v}{2\Delta^2}\right)
=
m-\frac{M_2(G)}{2\Delta^2}.
\end{equation}
Combining \eqref{eqn12} and \eqref{eqn15} yields
\[
DSO(G)^2 \le m\left(m-\frac{M_2(G)}{2\Delta^2}\right).
\]
We now examine the equality condition.
Suppose equality holds.
Then both inequalities \eqref{eqn12} and \eqref{eqn14}
must be equalities.
In particular, equality in \eqref{eqn14} yields
$d_u+d_v=2\Delta$ for every edge.
Since $d_u,d_v\le\Delta$, we obtain
$d_u=d_v=\Delta$ for every edge,
and hence $G$ is regular.
Conversely, if $G$ is regular, then  $d_u=d_v=\Delta$  for every edge $uv$. Hence $M_2(G)=m\Delta^2$ and $DSO(G)=\frac{m}{\sqrt{2}}$. The right-hand side becomes
\[
\sqrt{m^2-\frac{m}{2\Delta^2} \cdot m\Delta^2}
=
\frac{m}{\sqrt{2}},
\]
so equality holds. This completes the proof.
\end{proof}
\begin{theorem}\label{DSO-R}
Let $G$ be a graph with at least one edge.
Then
\[
\frac{\delta}{\sqrt{2}}\,R(G) \;\le\; DSO(G) \;\le\; \frac{\Delta}{\sqrt{2}}\,R(G).
\]
Equality holds in both bounds if and only if $G$ is regular.
\end{theorem}

\begin{proof}
\textbf{Lower bound.}
We first prove an edgewise inequality. For any edge $uv\in E(G)$,
\begin{equation}\label{eq03}
\frac{\sqrt{d_u^2+d_v^2}}{d_u+d_v} \ge \frac{\delta}{\sqrt{2}}\cdot\frac{1}{\sqrt{d_u d_v}}. 
\end{equation}
Since both sides are positive, squaring preserves the order.  Inequality \eqref{eq03} is equivalent to
\[
 2d_u d_v(d_u^2+d_v^2) \ge \delta^2(d_u+d_v)^2.
\]
Because $d_u,d_v\ge \delta$, we have $d_u d_v\ge \delta^2$, and therefore
\[
2d_u d_v(d_u^2+d_v^2) \ge 2\delta^2(d_u^2+d_v^2).
\]
Now,
\[
2\delta^2(d_u^2+d_v^2) - \delta^2(d_u+d_v)^2
= \delta^2\bigl[2(d_u^2+d_v^2) - (d_u+d_v)^2\bigr]
= \delta^2(d_u-d_v)^2 \ge 0.
\]
Thus
\[
2d_u d_v(d_u^2+d_v^2) \ge 2\delta^2(d_u^2+d_v^2) \ge \delta^2(d_u+d_v)^2,
\]
which establishes \eqref{eq03}.

Equality in the chain of inequalities requires
$d_u d_v = \delta^2$ and $(d_u-d_v)^2=0$, which together force $d_u = d_v = \delta$.
Now sum \eqref{eq03} over all edges:
\[
DSO(G) = \sum_{uv\in E(G)}\frac{\sqrt{d_u^2+d_v^2}}{d_u+d_v}
\ge \frac{\delta}{\sqrt{2}}\sum_{uv\in E(G)}\frac{1}{\sqrt{d_u d_v}}
= \frac{\delta}{\sqrt{2}}\,R(G).
\]

If equality holds in the theorem, then equality must hold in every instance of \eqref{eq03}, so every edge joins two vertices of degree $\delta$.  Because $G$ is connected, this implies that all vertices have degree $\delta$; hence $G$ is $\delta$-regular.
Conversely, if $G$ is regular of degree $r$, then $\delta = r$ and $R(G) = \frac{m}{r}$, while $DSO(G) = \frac{m}{\sqrt{2}}$.  The right‑hand side becomes $\frac{r}{\sqrt{2}}\cdot\frac{m}{r} = \frac{m}{\sqrt{2}}$, so equality is attained.  This completes the proof.

\noindent
\textbf{Upper bound.}
For an arbitrary edge $uv$, set $a = d_u/\Delta$ and $b = d_v/\Delta$.  Since $1 \le d_u,d_v \le \Delta$, we have $0 < a,b \le 1$.  The desired edgewise inequality
\[
\frac{\sqrt{d_u^2+d_v^2}}{d_u+d_v} \le \frac{\Delta}{\sqrt{2}}\,\frac{1}{\sqrt{d_u d_v}},
\]
is, after substituting $d_u = a\Delta$, $d_v = b\Delta$ and simplifying, equivalent to
\begin{equation}\label{eq04}
a b (a^2+b^2) \le \frac12 (a+b)^2. 
\end{equation}
Because $a,b \in (0,1]$, we have $a^2 \le a$ and $b^2 \le b$; hence $a^2+b^2 \le a+b$.
Multiplying this by the non‑negative number $ab$ yields
\[
ab(a^2+b^2) \le ab(a+b).
\]
Now, by the AM–GM inequality, $2ab \le a+b$, so $ab(a+b) \le \frac12 (a+b)^2$.
Chaining these inequalities gives exactly \eqref{eq04}.

Equality in \eqref{eq04} requires equality in both intermediate steps.
Equality in $a^2+b^2 \le a+b$ forces $a^2=a$ and $b^2=b$, which with $a,b>0$ implies $a=b=1$.
Equality in $2ab \le a+b$ also requires $a=b$.  Thus equality holds precisely when $a=b=1$, i.e., $d_u = d_v = \Delta$.

Summing the edgewise inequality over all edges, we obtain
\[
DSO(G) = \sum_{uv\in E(G)} \frac{\sqrt{d_u^2+d_v^2}}{d_u+d_v}
      \le \frac{\Delta}{\sqrt{2}} \sum_{uv\in E(G)} \frac{1}{\sqrt{d_u d_v}}
      = \frac{\Delta}{\sqrt{2}}\,R(G).
\]

If equality holds in the upper bound, then every edge must satisfy $d_u = d_v = \Delta$.  Since $G$ is connected and every vertex is incident with at least one edge, all vertices have degree $\Delta$; hence $G$ is $\Delta$-regular.
Conversely, if $G$ is regular of degree $r$, then $\Delta = r$ and $R(G) = m/r$, while $DSO(G) = m/\sqrt{2}$.  The right‑hand side becomes $\frac{r}{\sqrt{2}}\cdot\frac{m}{r} = \frac{m}{\sqrt{2}}$, so equality is achieved.
\end{proof}

\begin{theorem}\label{DSO-M1}
Let $G$ be a graph with $m>0$ edges. Then
\[
\frac{\delta\sqrt{2}\,m^2}{M_1(G)} \le DSO(G) \le m\sqrt{\frac{M_1(G)-m}{M_1(G)}}.
\]
Equality in the lower bound holds if and only if $G$ is regular.
Equality in the upper bound holds if and only if $G\cong K_2$.
\end{theorem}

\begin{proof}
\textbf{Lower bound.}
From Theorem~\ref{DSO-R}, we have
\[
DSO(G) \ge \frac{\delta}{\sqrt{2}} R(G),
\]
with equality if and only if every edge joins two vertices of degree $\delta$ (i.e., $G$ is $\delta$-regular).
By the AM--GM inequality, for every edge $uv$,
\[
\frac{1}{\sqrt{d_u d_v}} \ge \frac{2}{d_u+d_v},
\]
with equality if and only if $d_u=d_v$.
Summing over all edges yields
\[
R(G) \ge H(G),
\]
where $H(G)=\sum_{uv} \frac{2}{d_u+d_v}$, and equality holds if and only if $d_u=d_v$ for all edges.
Finally, by the Cauchy--Schwarz inequality,
\[
m^2 = \Bigl(\sum_{uv\in E(G)} 1\Bigr)^2
\le \Bigl(\sum_{uv\in E(G)} (d_u+d_v)\Bigr)
\Bigl(\sum_{uv\in E(G)} \frac{1}{d_u+d_v}\Bigr)
= M_1(G)\cdot \frac{H(G)}{2},
\]
so $H(G) \ge \frac{2m^2}{M_1(G)}$, with equality if and only if $d_u+d_v$ is constant over all edges.

Chaining these inequalities gives
\[
DSO(G) \ge \frac{\delta}{\sqrt{2}} R(G)
\ge \frac{\delta}{\sqrt{2}} H(G)
\ge \frac{\delta}{\sqrt{2}}\cdot \frac{2m^2}{M_1(G)}
= \frac{\delta\sqrt{2}\,m^2}{M_1(G)}.
\]
Equality in the final lower bound requires equality in each step, which happens precisely when $G$ is regular (all degrees equal to $\delta$). Conversely, if $G$ is regular of degree $r$, then $\delta=r$, $M_1(G)=2mr$, and the lower bound becomes $\frac{r\sqrt{2}\,m^2}{2mr} = \frac{m}{\sqrt{2}}$, which equals $DSO(G)$.

\noindent
\textbf{Upper bound.}
For each edge $uv\in E(G)$, set
$
x_{uv}=\frac{\sqrt{d_u^2+d_v^2}}{d_u+d_v}.
$
By the Cauchy–Schwarz inequality,
\[
DSO(G)^2 = \left(\sum_{uv\in E(G)} x_{uv}\right)^2
\le m \sum_{uv\in E(G)} x_{uv}^2.
\]

We claim that for every edge $uv$,
\[
x_{uv}^2 \le 1-\frac{1}{d_u+d_v}.
\]
Indeed, this is equivalent to
\[
\frac{d_u^2+d_v^2}{(d_u+d_v)^2} \le \frac{d_u+d_v-1}{d_u+d_v},
\]
which simplifies to
\[
d_u^2+d_v^2 \le (d_u+d_v)(d_u+d_v-1),\]
or equivalently, $0 \le 2d_ud_v-d_u-d_v.$
The last inequality holds because $d_u,d_v\ge 1$, so  the claim is proved.
Summing over all edges yields
\[
\sum_{uv\in E(G)} x_{uv}^2
\le m - \sum_{uv\in E(G)}\frac{1}{d_u+d_v}
= m - \frac{H(G)}{2}.
\]

By the Cauchy–Schwarz inequality,
\[
m^2 = \left(\sum_{uv\in E(G)} 1\right)^2
\le \left(\sum_{uv\in E(G)}(d_u+d_v)\right)
\left(\sum_{uv\in E(G)}\frac{1}{d_u+d_v}\right)
\]
\[
= M_1(G)\cdot \frac{H(G)}{2},
\]
so $\frac{H(G)}{2} \ge \frac{m^2}{M_1(G)}$.
Therefore,
\[
\sum_{uv\in E(G)} x_{uv}^2 \le m - \frac{m^2}{M_1(G)}
= \frac{m(M_1(G)-m)}{M_1(G)}.
\]
Hence,
\[
DSO(G)^2
\le m \cdot \frac{m(M_1(G)-m)}{M_1(G)}
= \frac{m^2(M_1(G)-m)}{M_1(G)},
\]
which gives the desired upper bound:
\[
DSO(G) \le m\sqrt{\frac{M_1(G)-m}{M_1(G)}}.
\]
Equality in the final inequality requires equality in all three preceding inequalities.
Equality in
$
x_{uv}^2 \le 1-\frac1{d_u+d_v}
$
is equivalent to
$
2d_ud_v-d_u-d_v=0,
$
whose only positive integer solution is
$
d_u=d_v=1.
$
Hence every edge joins two vertices of degree $1$, so $G\cong K_2$. In this case, $x_{uv}=1/\sqrt2$ and $d_u+d_v=2$ are constant on all edges, so equality also holds in both applications of the Cauchy--Schwarz inequality.
Conversely, for $G\cong K_2$, then $
M_1(G)=2$, $m=1$ and $DSO(G)=\frac1{\sqrt2}$,
and the upper bound is attained.
\end{proof}

\begin{lemma}\label{edge}
Let $a,b,c$ be real numbers with $a,b \ge c > 0$. Then
\[
a^2+b^2 \le \frac{(a+b)^3}{4c},
\]
with equality if and only if $a=b=c$.
\end{lemma}

\begin{proof}
Set $s=a+b$ and $p=ab$, so that $a^2+b^2=s^2-2p$. The claimed inequality is equivalent to
\begin{equation}\label{eqn101}
8c p \ge s^2(4c - s). 
\end{equation}
If $s\ge 4c$, the right-hand side of \eqref{eqn101} is non-positive while $p\ge c^2>0$, so \eqref{eqn101} holds trivially.

Now suppose $2c\le s<4c$ (recall $s\ge 2c$, since $a,b\ge c$). Viewing $p(a)=a(s-a)$ as a function of $a\in[c,s-c]$, this is a concave (downward) parabola, so it attains its minimum at the endpoints, giving
$
p\ge c(s-c).
$
It therefore suffices to verify \eqref{eqn101} for $p=c(s-c)$, i.e.
\[
8c^2(s-c)\ge s^2(4c-s) = 4c s^2 - s^3,
\]
which rearranges to
$
s^3 - 4c s^2 + 8c^2 s - 8c^3 \ge 0.
$
This polynomial factors as
\[
s^3-4c s^2+8c^2 s-8c^3 = (s-2c)\bigl((s-c)^2+3c^2\bigr),
\]
as may be checked directly by expansion. Since $s\ge 2c$, the factor $(s-2c)\ge0$, and $(s-c)^2+3c^2>0$ always; hence the product is non-negative, proving \eqref{eqn101} in this case as well.

For equality, both factors of the last expression must vanish or the product equal $0$ exactly at $s=2c$ (since the second factor is always strictly positive), and equality in $p\ge c(s-c)$ requires $a=c$ or $b=c$ (boundary of the concave function). Combining $s=2c$ with $a,b\ge c$ forces $a=b=c$. Conversely, if $a=b=c$, direct substitution shows both sides of the original inequality equal $8c^3$, so equality holds.
\end{proof}

\begin{theorem}\label{M1new}
Let $G$ be a graph with $m>0$ edges and minimum degree $\delta \ge 1$. Then
\[DSO(G) \le \sqrt{\frac{m\,M_1(G)}{4\delta}}.
\]
Equality holds if and only if $G$ is a regular graph.
\end{theorem}
\begin{proof}
For each edge $uv\in E(G)$, set
$
x_{uv} = \frac{\sqrt{d_u^2+d_v^2}}{d_u+d_v}.
$
By Cauchy–Schwarz,
\[
DSO(G)^2 = \left(\sum_{uv\in E(G)} x_{uv}\right)^2
\le m \sum_{uv} x_{uv}^2.
\]

Applying Lemma~\ref{edge} with $a=d_u$, $b=d_v$, and $c=\delta$, we obtain
\[
x_{uv}^2 = \frac{d_u^2+d_v^2}{(d_u+d_v)^2}
\le \frac{d_u+d_v}{4\delta},
\]
with equality if and only if $d_u=d_v=\delta$. Summing over all edges gives
\[
\sum_{uv\in E(G)} x_{uv}^2 \le \frac{1}{4\delta} \sum_{uv\in E(G)} (d_u+d_v)
= \frac{M_1(G)}{4\delta}.
\]
Hence
\[
DSO(G)^2 \le \frac{m\,M_1(G)}{4\delta}.
\]

For equality, both the Cauchy--Schwarz inequality and Lemma~\ref{edge}
must be equalities.
Equality in Lemma~\ref{edge} implies
$
d_u=d_v=\delta
$
for every edge $uv\in E(G)$.
Since $\delta\ge1$, every vertex is incident with an edge, and therefore every vertex has degree $\delta$.
Hence $G$ is regular.

Conversely, if $G$ is regular, then $d_u=d_v=\delta$ for every edge, so equality holds in Lemma~\ref{edge}. Moreover,
$
x_{uv}=\frac1{\sqrt2}
$
for every edge, and therefore equality also holds in the Cauchy--Schwarz inequality. Hence equality is attained.
\end{proof}

\begin{theorem}\label{t3.6}
Let $G$ be a nontrivial graph. Then
\[
\frac{1}{\sqrt{2}}\,GA(G) \le DSO(G) \le \frac{1}{2}\sqrt{\frac{\Delta}{\delta}+\frac{\delta}{\Delta}}\; GA(G).
\]
Equality holds in the lower bound if and only if $G$ is a regular graph.  
Equality holds in the upper bound if and only if either the graph is regular or $G$ is bipartite semiregular with part degrees $\Delta$ and $\delta$.
\end{theorem}

\begin{proof}
Using the identity
\[
\frac{\sqrt{d_u^2+d_v^2}}{d_u+d_v} = \sqrt{\frac{d_u}{d_v}+\frac{d_v}{d_u}}\;\frac{\sqrt{d_u d_v}}{d_u+d_v},
\]
we bound the square root factor. The inequality $t+1/t \ge 2$ for $t>0$ gives
\[
\sqrt{\frac{d_u}{d_v}+\frac{d_v}{d_u}} \ge \sqrt{2},
\]
with equality if and only if $d_u/d_v = 1$, i.e., $d_u=d_v$. Hence,
\[
DSO(G) \ge \sqrt{2} \sum_{uv\in E(G)} \frac{\sqrt{d_u d_v}}{d_u+d_v}
= \frac{\sqrt{2}}{2} GA(G),
\]
proving the lower bound. Equality holds precisely when $d_u=d_v$ for every edge, i.e., $G$ is regular.

For the upper bound, assume without loss of generality $d_u \ge d_v$ and set $t = d_u/d_v$. Then $1\le t \le \Delta/\delta$. The function $f(t)=t+1/t$ is increasing for $t\ge 1$, so
\[
\sqrt{\frac{d_u}{d_v}+\frac{d_v}{d_u}} \le \sqrt{\frac{\Delta}{\delta}+\frac{\delta}{\Delta}},
\]
with equality for an edge if and only if either $\Delta=\delta$ or $t = \Delta/\delta$ (i.e., $\{d_u,d_v\} = \{\Delta,\delta\}$).  
Thus,
\[
DSO(G) \le \sqrt{\frac{\Delta}{\delta}+\frac{\delta}{\Delta}} \sum_{uv\in E(G)} \frac{\sqrt{d_u d_v}}{d_u+d_v}
= \frac{1}{2}\sqrt{\frac{\Delta}{\delta}+\frac{\delta}{\Delta}}\; GA(G).
\]
For the sum to achieve equality, every edge must attain equality in the pointwise factor. Hence if $\Delta>\delta$, every edge must join a vertex of degree $\Delta$ with a vertex of degree $\delta$. 
In this case, $G$ is a bipartite semiregular graph with part degrees $\Delta$ and $\delta$.
If $\Delta=\delta$, the factor equals $\sqrt{2}$ for every edge automatically, and equality holds (the graph is regular). This completes the proof.
\end{proof}
\begin{remark}
Bounds relating $DSO(G)$ to the geometric--arithmetic index $GA(G)$ also appear in \cite{5}, where it is shown that
\[
\frac{\sqrt{2}}{4\Delta}Alb(G)+\frac12 GA(G) \le DSO(G) \le \frac{1}{2\delta}Alb(G)+\frac{\sqrt{2}}{2}GA(G).
\]
A direct comparison shows that neither bound dominates the other uniformly; the two are complementary, each being sharper on a different extremal family.

For a regular graph of degree $r$ with $m$ edges, $Alb(G)=0$ and $DSO(G)=m/\sqrt2$. The lower bound of Theorem~\ref{t3.6} attains this value exactly,
\[
 (1/\sqrt2)GA(G)=m/\sqrt2, 
 \]
 whereas the lower bound above reduces to $\tfrac12 GA(G)=m/2$, which is strictly weaker. Both upper bounds coincide and are exact in this case, since $Alb(G)=0$.

The situation is different for the lower bound on bipartite semiregular graphs.
Writing $DSO(G)=m\sqrt{\Delta^2+\delta^2}/(\Delta+\delta)$, the lower bound of Theorem~\ref{t3.6} gives $(1/\sqrt2)GA(G)=m\sqrt{2\Delta\delta}/(\Delta+\delta)$, which by $\Delta^2+\delta^2\ge 2\Delta\delta$ is in general farther from the true value than the lower bound above, which incorporates the additive $Alb(G)$ term and tracks the degree imbalance more closely (e.g.\ for $\Delta=4,\delta=1$, the bound above yields $\approx 0.665\,m$ versus $\approx 0.566\,m$ from Theorem~\ref{t3.6}, against the true value $\approx 0.825\,m$). 

On the other hand, the upper bound of Theorem~\ref{t3.6} is attained with equality on this family, $\tfrac12\sqrt{\Delta/\delta+\delta/\Delta}\,GA(G)=DSO(G)$, while the upper bound above, $\tfrac1{2\delta}Alb(G)+\tfrac{\sqrt2}{2}GA(G)$, is considerably looser (e.g.\ $\approx 2.07\,m$ for the same example).

In summary, Theorem~\ref{t3.6} provides a sharper, tight characterization for the upper bound on both regular and bipartite semiregular graphs, while the additive bound of \cite{5} is sharper for the lower bound on graphs with large degree imbalance. The two results are therefore best viewed as complementary rather than one subsuming the other.
\end{remark}

\begin{lemma}\label{const}
Let $G$ be a graph.
If $d_u^2+d_v^2$ is constant for all edges $uv\in E(G)$,
then $G$ is either a regular graph or a bipartite semiregular graph.
\end{lemma}

\begin{proof}
Suppose $d_u^2+d_v^2 = c$ for every edge $uv$.
Let $v$ be a vertex of degree $r$, and let $u$ be any neighbor of $v$.
Then $d_u^2+r^2 = c$, so $d_u = \sqrt{c-r^2}$.
Thus every neighbor of a vertex of degree $r$ has the same degree
$s = \sqrt{c-r^2}$.
Similarly, every neighbor of a vertex of degree $s$ has degree $r$.
Because $G$ is connected, only the two degrees $r$ and $s$ can occur in $G$.
If $r=s$, then $G$ is regular.
If $r\neq s$, then every edge joins a vertex of degree $r$ to a vertex of degree $s$;
hence the vertices of degree $r$ and those of degree $s$ form the two parts of a bipartition,
and $G$ is bipartite semiregular.
\end{proof}

\begin{theorem}\label{DSO-H-ISI-M1}
Let $G$ be a  graph. Then
\[
DSO(G)\le \sqrt{\frac{H(G)}{2} \left(M_1(G)-2ISI(G)\right)}.
\]
Equality holds if and only if $G$ is either a regular graph or a bipartite semiregular graph.
\end{theorem}

\begin{proof}
For every edge $uv\in E(G)$, define
$
a_{uv}=\frac{\sqrt{d_u^2+d_v^2}}{\sqrt{d_u+d_v}}$ and $b_{uv}=\frac{1}{\sqrt{d_u+d_v}}$.
Then $a_{uv} b_{uv}=\frac{\sqrt{d_u^2+d_v^2}}{d_u+d_v}$.
Applying the Cauchy–Schwarz inequality to the sequences $a_{uv}$ and $b_{uv}$ gives
\begin{equation}\label{eqn16}
DSO(G)^2
\le 
\left(\sum_{uv\in E(G)} \frac{d_u^2+d_v^2}{d_u+d_v}\right)
\left(\sum_{uv\in E(G)} \frac{1}{d_u+d_v}\right). 
\end{equation}
Since $d_u^2+d_v^2=(d_u+d_v)^2-2d_u d_v$, we have
\[
\frac{d_u^2+d_v^2}{d_u+d_v}
=
(d_u+d_v)-\frac{2d_u d_v}{d_u+d_v}.
\]
Thus
\begin{equation}\label{eqn17}
\sum_{uv\in E(G)} \frac{d_u^2+d_v^2}{d_u+d_v}
=
M_1(G)-2ISI(G).
\end{equation}
Moreover, from the definition of $H(G)$,
\begin{equation}\label{eqn18}
\sum_{uv\in E(G)} \frac{1}{d_u+d_v}
=
\frac{H(G)}{2}. 
\end{equation}
Substituting \eqref{eqn17} and \eqref{eqn18} into \eqref{eqn16} yields
\[
DSO(G)^2
\le 
\frac{H(G)}{2} \left(M_1(G)-2ISI(G)\right),
\]
hence
\[
DSO(G)
\le 
\sqrt{\frac{H(G)}{2} \left(M_1(G)-2ISI(G)\right)}.
\]
Equality holds if and only if the Cauchy–Schwarz inequality \eqref{eqn16} is an equality. Since $a_{uv}, b_{uv}\ge 0$, this occurs precisely when there exists a constant $k$ such that $a_{uv}=k\,b_{uv}$ for every $uv\in E(G)$.
Now
$
\frac{a_{uv}}{b_{uv}} = \sqrt{d_u^2+d_v^2},
$
so this is equivalent to
$d_u^2+d_v^2$
being constant over all edges.
 By Lemma~\ref{const}, $G$ is therefore regular or bipartite semiregular.

Conversely, suppose $G$ is regular or bipartite semiregular.
In a regular graph of degree $r$, we have
$M_1(G)=2mr$, $H(G)=m/r$, $ISI(G)=mr/2$, and $DSO(G)=m/\sqrt{2}$,
so the right‑hand side becomes
$ m/\sqrt{2}$, giving equality.
If $G$ is bipartite semiregular with part degrees $\Delta$ and $\delta$, then
$M_1(G)=m(\Delta+\delta)$, $H(G)=2m/(\Delta+\delta)$,
$ISI(G)=m\Delta\delta/(\Delta+\delta)$, and
$DSO(G)=m\sqrt{\Delta^2+\delta^2}/(\Delta+\delta)$.
A short calculation shows
$M_1(G)-2ISI(G)=m(\Delta^2+\delta^2)/(\Delta+\delta)$,
and the right‑hand side becomes
\[\sqrt{\frac{H(G)}{2}(M_1(G)-2ISI(G))}
= m\frac{\sqrt{\Delta^2+\delta^2}}{(\Delta+\delta)}=DSO(G).\]
Hence equality holds in both cases.
\end{proof}
\begin{remark}
A related lower bound involving the same combination $M_1(G)-2ISI(G)$ together with the harmonic index $H(G)$ was obtained independently by Movahedi \cite{5} as follows:
\[
DSO(G)\ge \frac{M_1(G)-2ISI(G)+\delta\Delta H(G)}{\sqrt{2}(\Delta+\delta)}.
\]

Moreover, the upper bound presented in Theorem~\ref{DSO-H-ISI-M1} is a direct sharpening of the upper bound from Theorem 3.4 of \cite{5}:
\[
DSO(G)\le \frac{\sqrt{2}}{2}\sqrt{H(G)\bigl(M_1(G)-H(G)\bigr)}.
\]
Indeed, since $d_ud_v \ge 1$ for every edge, we have
\[
2ISI(G)=\sum_{uv\in E(G)}\frac{2d_ud_v}{d_u+d_v}
\ge \sum_{uv\in E(G)}\frac{2}{d_u+d_v}=H(G).
\]
Hence $M_1(G)-2ISI(G) \le M_1(G)-H(G)$, which immediately yields
\[
\sqrt{\frac{H(G)}{2}\bigl(M_1(G)-2ISI(G)\bigr)}
\le
\sqrt{\frac{H(G)}{2}\bigl(M_1(G)-H(G)\bigr)}.
\]
Thus the new bound is always at least as sharp as the earlier one, with equality (for connected graphs) occurring only when $G\simeq K_2$.
\end{remark}
\begin{lemma}\label{newl}
For every $t\in[0,1]$, we have
$
\sqrt{1+t}\ge1+\frac{t}{4},
$
with equality if and only if $t=0$.
\end{lemma}

\begin{proof}
Let
$
f(t)=\sqrt{1+t}-1-\frac{t}{4}.
$
Since
$
f''(t)= -\frac1{4(1+t)^{3/2}}<0,
$
the function $f$ is concave on $[0,1]$. Also,
$
f(0)=0$ and $f(1)=\sqrt2-\frac54>0$.
Therefore $f(t)\ge0$ on $[0,1]$, and equality holds only at $t=0$.
\end{proof}

\begin{theorem}\label{new:Alb-upper}
Let $G$ be a graph with $m>0$ edges. Then
\[
\frac{m}{\sqrt{2}}+\frac{Alb(G)^2}{16\sqrt{2}\,m \Delta^2}\le DSO(G) \le \frac{m}{\sqrt{2}} + \frac{Alb(G)}{4\sqrt{2}\,\delta}.
\]
Equality holds if and only if $G$ is regular.
\end{theorem}
\begin{proof}
\textbf{Lower bound.}
From
\[
d_u^2+d_v^2=\frac{(d_u+d_v)^2+(d_u-d_v)^2}{2},
\]
we get
\begin{equation}\label{eqn19}
\frac{\sqrt{d_u^2+d_v^2}}{d_u+d_v}
=
\frac{1}{\sqrt{2}}
\sqrt{1+\left(\frac{d_u-d_v}{d_u+d_v}\right)^2}.
\end{equation}
Hence
\begin{equation}\label{eqn20}
DSO(G)
=
\frac{1}{\sqrt{2}}
\sum_{uv\in E(G)}
\sqrt{1+\left(\frac{d_u-d_v}{d_u+d_v}\right)^2}.
\end{equation}
By Lemma \ref{newl} for every $t\in[0,1]$, we have
\begin{equation}\label{eqn21}
\sqrt{1+t}\ge 1+\frac{t}{4}, 
\end{equation}
with equality if and only if $t=0$.
Also,
$
0\le \left(\frac{d_u-d_v}{d_u+d_v}\right)^2 \le 1.
$
Applying \eqref{eqn21} to \eqref{eqn20} gives
\[
DSO(G)
\ge 
\frac{1}{\sqrt{2}}
\sum_{uv\in E(G)}
\left(
1+\frac{1}{4}
\left(\frac{d_u-d_v}{d_u+d_v}\right)^2
\right).
\]
So
\begin{equation}\label{eqn22}
DSO(G)
\ge 
\frac{m}{\sqrt{2}}
+
\frac{1}{4\sqrt{2}}
\sum_{uv\in E(G)}
\frac{(d_u-d_v)^2}{(d_u+d_v)^2}. 
\end{equation}
Since $d_u+d_v\le 2\Delta$ for every edge,
\[
\frac{(d_u-d_v)^2}{(d_u+d_v)^2}
\ge 
\frac{(d_u-d_v)^2}{4\Delta^2},
\]
with equality whenever $d_u-d_v=0$ (i.e., both sides are zero) or $d_u+d_v=2\Delta$. Substituting into \eqref{eqn22} yields
\begin{equation}\label{eqn23}
DSO(G)
\ge 
\frac{m}{\sqrt{2}}
+
\frac{1}{16\sqrt{2}\,\Delta^2}
\sum_{uv\in E(G)} (d_u-d_v)^2. 
\end{equation}
By Cauchy–Schwarz,
\[
\left(\sum_{uv\in E(G)} |d_u-d_v|\right)^2
\le 
m\sum_{uv\in E(G)} (d_u-d_v)^2,
\]
with equality if and only if all $|d_u-d_v|$ are equal.
Hence
\begin{equation}\label{eqn24}
\sum_{uv\in E(G)} (d_u-d_v)^2
\ge 
\frac{Alb(G)^2}{m}. 
\end{equation}
Combining \eqref{eqn23} and \eqref{eqn24} gives the desired inequality.

Now suppose equality holds. Then inequalities \eqref{eqn21}, the step leading to \eqref{eqn23}, and \eqref{eqn24} must all be equalities.
Equality in \eqref{eqn21} forces $\left(\frac{d_u-d_v}{d_u+d_v}\right)^2=0$ for every edge, i.e., $d_u=d_v$ for all edges.
Equality in the step $ \frac{(d_u-d_v)^2}{(d_u+d_v)^2} \ge \frac{(d_u-d_v)^2}{4\Delta^2} $ is automatic when $d_u-d_v=0$ (both sides are zero), regardless of whether $d_u+d_v=2\Delta$. Thus it imposes no extra condition beyond $d_u=d_v$.
Equality in \eqref{eqn24} is automatic because all $|d_u-d_v|$ are zero, making both sides zero.
Hence equality holds precisely when $d_u=d_v$ for every edge. Since $G$ is connected, this forces all vertices to have the same degree, i.e., $G$ is regular. Conversely, if $G$ is regular of degree $r$, then $Alb(G)=0$ and $DSO(G)=m/\sqrt{2}$; the right-hand side becomes $m/\sqrt{2}+0 = m/\sqrt{2}$, so equality is attained.

\noindent
\textbf{Upper bound.}
Using the identity
\[
\frac{\sqrt{d_u^2+d_v^2}}{d_u+d_v}
= \frac{1}{\sqrt{2}}\sqrt{1+\left(\frac{d_u-d_v}{d_u+d_v}\right)^2},
\]
and the elementary inequality $\sqrt{1+t}\le 1+\frac{t}{2}$ for $t\ge0$, we get
\[
DSO(G) \le \frac{m}{\sqrt{2}} + \frac{1}{2\sqrt{2}}\sum_{uv\in E(G)} \frac{(d_u-d_v)^2}{(d_u+d_v)^2}.
\]
Since $d_u+d_v \ge 2\delta$, we have $\frac{1}{(d_u+d_v)^2} \le \frac{1}{4\delta^2}$. Also, $|d_u-d_v| \le d_u+d_v$, hence $(d_u-d_v)^2 \le (d_u+d_v)|d_u-d_v|$. Therefore,
\[
\frac{(d_u-d_v)^2}{(d_u+d_v)^2} \le \frac{|d_u-d_v|}{d_u+d_v} \le \frac{|d_u-d_v|}{2\delta}.
\]
Substituting this into the previous inequality gives
\[
DSO(G) \le \frac{m}{\sqrt{2}} + \frac{1}{2\sqrt{2}}\sum_{uv\in E(G)} \frac{|d_u-d_v|}{2\delta}
= \frac{m}{\sqrt{2}} + \frac{Alb(G)}{4\sqrt{2}\,\delta}.
\]
Suppose equality holds.
Then equality must occur in every inequality used above.
In particular, equality in
$
\sqrt{1+t}\le1+\frac t2
$
forces
$t=0$,
that is,
$d_u=d_v$
for every edge.
Hence all remaining inequalities are also equalities,
and, since $G$ is connected,
$G$ is regular.
Conversely, if $G$ is regular, then $Alb(G)=0$ and both sides equal $m/\sqrt{2}$.
\end{proof}

\begin{remark}
Comparing the upper bound obtained in Theorem~\ref{new:Alb-upper},
\[
DSO(G) \le \frac{m}{\sqrt{2}} + \frac{Alb(G)}{4\sqrt{2}\,\delta},
\]
with the previously known bound in \cite{5},
\[
DSO(G) \le \frac{\sqrt{2}}{2}\bigl(Alb(G)+m\bigr)
= \frac{m}{\sqrt{2}} + \frac{Alb(G)}{\sqrt{2}},
\]
we observe that our bound is strictly sharper for every non-regular graph, since $\delta \ge 1$ implies $4\sqrt{2}\,\delta \ge 4\sqrt{2} > \sqrt{2}$. In the regular case, both bounds reduce to the sharp value $m/\sqrt{2}$.

Regarding the lower bound, Theorem 3.1 of \cite{5} provides a mixed bound involving both the Albertson and the geometric-arithmetic indices:
\[
DSO(G) \ge \frac{\sqrt{2}}{4\Delta}Alb(G) + \frac{1}{2}GA(G).
\]
The lower bound presented here,
\[
DSO(G) \ge \frac{m}{\sqrt{2}} + \frac{Alb(G)^2}{16\sqrt{2}\,m\,\Delta^2},
\]
is of a complementary nature. It depends solely on $m$, $Alb$, and $\Delta$, and performs particularly well when $Alb(G)$ is large (e.g., for stars), whereas the bound from \cite{5} benefits from the additional $GA(G)$ term in near-regular graphs. Thus neither lower bound uniformly dominates the other, and they are best viewed as complementary tools.
\end{remark}

\begin{theorem}\label{alt}
Let $G$ be a  graph with $m>0$ edges. Then
\[
DSO(G) \le \sqrt{\frac{m^{2}}{2}+\frac{m}{8}\bigl(F(G)-2M_{2}(G)\bigr)}.
\]
Equality holds if and only if $G$ is regular.
\end{theorem}
\begin{proof}[Proof of the alternative bound]
By the Cauchy--Schwarz inequality,
\[
DSO(G)^{2} \le m\sum_{uv\in E(G)} \frac{d_u^{2}+d_v^{2}}{(d_u+d_v)^{2}}.
\]
For each edge,
\[
\frac{d_u^{2}+d_v^{2}}{(d_u+d_v)^{2}}
= \frac12+\frac12\cdot\frac{(d_u-d_v)^{2}}{(d_u+d_v)^{2}}.
\]
Summation over all edges yields
\[
\sum_{uv\in E(G)}\frac{d_u^{2}+d_v^{2}}{(d_u+d_v)^{2}}
= \frac{m}{2}+\frac12\sum_{uv\in E(G)}\frac{(d_u-d_v)^{2}}{(d_u+d_v)^{2}}.
\]
Because $d_u,d_v\ge 1$, we have $d_u+d_v\ge 2$ and thus $(d_u+d_v)^{2}\ge 4$.
Consequently,
\[
\frac{(d_u-d_v)^{2}}{(d_u+d_v)^{2}} \le \frac{(d_u-d_v)^{2}}{4}
\]
for every edge. Summing gives
\[
\sum_{uv\in E(G)}\frac{(d_u-d_v)^{2}}{(d_u+d_v)^{2}}
\le \frac14\sum_{uv\in E(G)}(d_u-d_v)^{2}
= \frac{\sigma(G)}{4}.
\]
Therefore,
\[
\sum_{uv\in E(G)}\frac{d_u^{2}+d_v^{2}}{(d_u+d_v)^{2}}
\le \frac{m}{2}+\frac{\sigma(G)}{8}.
\]
Plugging this into the Cauchy--Schwarz bound,
\[
DSO(G)^{2} \le m\Bigl(\frac{m}{2}+\frac{\sigma(G)}{8}\Bigr)
= \frac{m^{2}}{2}+\frac{m}{8}\,\sigma(G).
\]
Now $\sigma(G)=\sum_{uv\in E(G)}(d_u^{2}+d_v^{2}-2d_ud_v)=F(G)-2M_{2}(G)$, so we obtain the claimed inequality.

\medskip
\noindent\textit{Equality analysis.}
Assume equality holds in the final bound.
Then equality must hold in both the Cauchy--Schwarz inequality and the estimate
\[
\frac{(d_u-d_v)^{2}}{(d_u+d_v)^{2}} \le \frac{(d_u-d_v)^{2}}{4} \qquad\text{for every edge}.
\]
Equality in
\[
\frac{(d_u-d_v)^2}{(d_u+d_v)^2}
\le
\frac{(d_u-d_v)^2}{4},
\]
holds if and only if either $d_u=d_v$ (when both sides are zero) or
$(d_u+d_v)^2=4$.
Since $d_u,d_v\ge1$, the latter also implies $d_u=d_v=1$.
Hence equality is equivalent to $d_u=d_v$.
Hence every edge joins vertices of equal degree. Since $G$ is connected,
all vertices have the same degree, and therefore $G$ is regular. For any such edge with $d_u=d_v=r$, we have
\[
\frac{\sqrt{d_u^{2}+d_v^{2}}}{d_u+d_v} = \frac{\sqrt{2r^{2}}}{2r} = \frac{1}{\sqrt{2}},
\]
a constant independent of the edge. Thus the terms in the Cauchy--Schwarz sum are all equal, and equality in the Cauchy--Schwarz inequality is automatically satisfied.
Consequently, equality holds precisely for graphs whose components are regular.
For such graphs $F(G)-2M_{2}(G)=0$, and the bound reduces to $m/\sqrt{2}$, which equals $DSO(G)$.
This completes the characterization.
\end{proof}

\begin{remark}\label{rem:comparison-upper}
We briefly compare the two upper bounds for $DSO(G)$ obtained in Theorems \ref{DSO-M1} and \ref{alt}.

The bound from Theorem~\ref{DSO-M1},
\[
DSO(G)\le m\sqrt{\frac{M_1(G)-m}{M_1(G)}},
\]
is sharp only for $K_2$. For an $r$-regular graph with $r>1$ we have $M_1(G)=2mr$, and the bound becomes $m\sqrt{1-\frac{1}{2r}}$, which tends to $m$ as $r\to\infty$, while the true value is $m/\sqrt2\approx0.707m$. Thus it is very weak for dense regular graphs.
On the other hand, for some irregular graphs such as stars it gives a reasonable estimate: e.g., for $K_{1,4}$ we have $m=4$, $M_1=20$, the bound evaluates to $4\sqrt{0.8}\approx3.578$, while the exact $DSO$ is $\approx3.298$.

The alternative bound derived in Theorem~\ref{alt}
\[
DSO(G)\le\sqrt{\frac{m^2}{2}+\frac{m}{8}\bigl(F(G)-2M_2(G)\bigr)},
\]
is exact for every regular graph, because then
$F(G)-2M_2(G)=0$
and the bound reduces to $m/\sqrt2$.
However, its quality deteriorates when the graph is far from regular. For the star $K_{1,4}$ it gives $\sqrt{26}\approx5.099$, which is significantly weaker than the bound from Theorem~\ref{DSO-M1}.
In general, the term $\frac{m}{8}(F-2M_2)=\frac{m}{8}\sigma(G)$ increases with the irregularity of the graph, so the bound becomes looser for highly irregular graphs.

In summary, neither bound uniformly dominates the other.
The bound of Theorem~\ref{DSO-M1} is often tighter for highly irregular graphs,
such as stars, whereas the alternative bound is exact for regular graphs and
typically performs better for graphs that are close to regular.
Thus the two bounds complement each other.
\end{remark}

\begin{theorem}\label{DSO-F}
Let $G$ be a graph with $m>0$ edges.
Then
\[
\frac{F(G)}{2\sqrt{2}\,\Delta^{2}} \;\le\; DSO(G) \;\le\; \frac{1}{2\delta}\sqrt{m\,F(G)}.
\]
Equality holds in both bounds if and only if $G$ is regular.
\end{theorem}

\begin{proof}
\noindent\textbf{Upper bound.}
By the Cauchy--Schwarz inequality,
\[
DSO(G)^{2} \le m\sum_{uv\in E(G)}\frac{d_u^{2}+d_v^{2}}{(d_u+d_v)^{2}}.
\]
Since $d_u+d_v \ge 2\delta$ for every edge, we have $\frac{1}{(d_u+d_v)^{2}} \le \frac{1}{4\delta^{2}}$.
Hence
\[
DSO(G)^{2} \le m\sum_{uv\in E(G)}\frac{d_u^{2}+d_v^{2}}{4\delta^{2}}
= \frac{m}{4\delta^{2}}\,F(G),
\]
and taking square roots gives the claimed upper bound.

Equality in the Cauchy--Schwarz step requires the terms $\frac{\sqrt{d_u^{2}+d_v^{2}}}{d_u+d_v}$ to be equal for all edges, while equality in $\frac{1}{(d_u+d_v)^{2}} \le \frac{1}{4\delta^{2}}$ forces $d_u+d_v = 2\delta$ for every edge.
Together, these imply $d_u = d_v = \delta$ for every edge, so every vertex has degree $\delta$; thus $G$ is $\delta$-regular.
Conversely, if $G$ is regular of degree $r$, then $\delta=r$, $F(G)=2mr^{2}$, and the right‑hand side becomes $\frac{1}{2r}\sqrt{m\cdot 2mr^{2}} = \frac{m}{\sqrt{2}} = DSO(G)$.

\medskip
\noindent\textbf{Lower bound.}
For a single edge $uv$ we have $d_u+d_v \le 2\Delta$ and $d_u^{2}+d_v^{2} \le 2\Delta^{2}$.
Since $d_u^{2}+d_v^{2} \le 2\Delta^{2}$, we have
\[
\sqrt{d_u^{2}+d_v^{2}} \ge \frac{d_u^{2}+d_v^{2}}{\sqrt{2\Delta^{2}}}
= \frac{d_u^{2}+d_v^{2}}{\sqrt{2}\,\Delta}.
\]
Together with $\frac{1}{d_u+d_v} \ge \frac{1}{2\Delta}$ this yields
\[
\frac{\sqrt{d_u^{2}+d_v^{2}}}{d_u+d_v}
\ge \frac{d_u^{2}+d_v^{2}}{2\sqrt{2}\,\Delta^{2}}.
\]
Summing over all edges gives
\[
DSO(G) \ge \frac{1}{2\sqrt{2}\,\Delta^{2}}\sum_{uv}(d_u^{2}+d_v^{2})
= \frac{F(G)}{2\sqrt{2}\,\Delta^{2}}.
\]

Equality in the lower bound requires both $d_u+d_v = 2\Delta$ and $d_u^{2}+d_v^{2} = 2\Delta^{2}$ for every edge, which together force $d_u = d_v = \Delta$.
Thus every edge joins two vertices of degree $\Delta$; as $\Delta\ge 1$, every vertex has degree $\Delta$ and $G$ is $\Delta$-regular.
Conversely, if $G$ is regular of degree $r=\Delta$, then $F(G)=2mr^{2}$ and the right‑hand side equals $\frac{2mr^{2}}{2\sqrt{2}\,r^{2}} = \frac{m}{\sqrt{2}} = DSO(G)$.
\end{proof}

\begin{remark}\label{rem:compare-B-C}
Both the bound of Theorem~\ref{DSO-F} and that of Theorem~\ref{Delta-delta} are never smaller than $m/\sqrt{2}$, with equality exactly for regular graphs.
Nevertheless, neither bound is universally stronger.
Theorem~\ref{Delta-delta} is exact for every bipartite semiregular graph (including stars), while Theorem~\ref{DSO-F} is exact only for regular graphs.
When the graph contains many edges joining vertices of degree $\Delta$ and $\delta$, the bound $m\frac{\sqrt{\Delta^{2}+\delta^{2}}}{\Delta+\delta}$ is close to the true $DSO$ and can be significantly smaller than $\frac{1}{2\delta}\sqrt{m\,F(G)}$; a star $K_{1,n-1}$ illustrates this sharply.
On the other hand, if most edges connect vertices whose degrees are near the minimum $\delta$, the factor $1/(2\delta)$ may give Theorem~\ref{DSO-F} a slight advantage.
For instance, a graph with $\delta=2$, $\Delta=3$ and a large number of $(2,2)$-edges yields a bound from Theorem~\ref{DSO-F} that is close to $0.7071m$ (approaching this value as the proportion of $(2,2)$-edges increases), while the bound from Theorem~\ref{Delta-delta} equals $0.7212m$.
\end{remark}

\begin{theorem}\label{thm:DSO-sigma}
Let $G$ be a graph with $m>0$ edges. Then
\[
\frac{m}{\sqrt{2}} + \frac{\sqrt{2}-1}{4\sqrt{2}} \cdot \frac{\sigma(G)}{\Delta^2}
\;\le\; DSO(G)
\;\le\; \frac{m}{\sqrt{2}} + \frac{\sigma(G)}{8\sqrt{2}\,\delta^2}.
\]
Equality holds in both bound if and only if $G$ is regular.
\end{theorem}

\begin{proof}
For each edge $uv\in E(G)$, set
$
t_{uv} = \left(\frac{d_u-d_v}{d_u+d_v}\right)^2.
$
Since $d_u,d_v\ge 1$, we have $0 \le t_{uv} < 1$. From the identity
\[
\frac{\sqrt{d_u^2+d_v^2}}{d_u+d_v}
= \frac{1}{\sqrt{2}}\sqrt{1+t_{uv}},
\]
we obtain
\[
DSO(G) = \frac{1}{\sqrt{2}} \sum_{uv\in E(G)} \sqrt{1+t_{uv}}.
\]

\medskip
\noindent \textbf{Upper bound.}  
Using the elementary inequality $\sqrt{1+t} \le 1+\frac{t}{2}$ for $t\ge 0$, we get
\[
DSO(G) \le \frac{m}{\sqrt{2}} + \frac{1}{2\sqrt{2}} \sum_{uv\in E(G)} \frac{(d_u-d_v)^2}{(d_u+d_v)^2}.
\]
Since $d_u+d_v \ge 2\delta$, we have $\frac{1}{(d_u+d_v)^2} \le \frac{1}{4\delta^2}$. Therefore,
\[
DSO(G) \le \frac{m}{\sqrt{2}} + \frac{1}{8\sqrt{2}\,\delta^2} \sum_{uv\in E(G)} (d_u-d_v)^2
= \frac{m}{\sqrt{2}} + \frac{\sigma(G)}{8\sqrt{2}\,\delta^2}.
\]

For equality in the upper bound, we need equality in both inequalities used. Equality in $\sqrt{1+t} \le 1+t/2$ occurs only when $t=0$, which gives $d_u=d_v$ for every edge. Equality in $\frac{1}{(d_u+d_v)^2} \le \frac{1}{4\delta^2}$ requires $d_u+d_v=2\delta$ for every edge. Together, these force $d_u=d_v=\delta$ for every edge, so $G$ is regular. Conversely, if $G$ is regular, then $\sigma(G)=0$ and the bound becomes $DSO(G)=m/\sqrt{2}$, which is equality.

\medskip
\noindent \textbf{Lower bound.}  
Since $0 \le t < 1$, the function $\sqrt{1+t}$ is concave on $[0,1]$, so its graph lies above the chord joining $(0,1)$ and $(1,\sqrt{2})$. Thus,
\[
\sqrt{1+t} \ge 1 + (\sqrt{2}-1)t.
\]
Therefore,
\[
DSO(G) \ge \frac{m}{\sqrt{2}} + \frac{\sqrt{2}-1}{\sqrt{2}} \sum_{uv} \frac{(d_u-d_v)^2}{(d_u+d_v)^2}.
\]
Using $d_u+d_v \le 2\Delta$, we have $\frac{1}{(d_u+d_v)^2} \ge \frac{1}{4\Delta^2}$, hence
\[
DSO(G) \ge \frac{m}{\sqrt{2}} + \frac{\sqrt{2}-1}{4\sqrt{2}} \cdot \frac{1}{\Delta^2} \sum_{uv} (d_u-d_v)^2
= \frac{m}{\sqrt{2}} + \frac{\sqrt{2}-1}{4\sqrt{2}} \cdot \frac{\sigma(G)}{\Delta^2}.
\]

For equality in the lower bound, we need equality in both inequalities. Equality in the chord inequality occurs at $t=0$ or $t=1$. The case $t=1$ would require $|d_u-d_v|=d_u+d_v$, which implies one of the degrees is zero; this is impossible because every vertex has degree at least 1. Hence we must have $t=0$, so $d_u=d_v$ for every edge. Equality in $\frac{1}{(d_u+d_v)^2} \ge \frac{1}{4\Delta^2}$ requires $d_u+d_v=2\Delta$ for every edge. Together, these give $d_u=d_v=\Delta$, so $G$ is regular. Conversely, if $G$ is regular, then $\sigma(G)=0$ and the bound becomes equality.
\end{proof}

\begin{remark}
The proof of Theorem~\ref{alt} yields the intermediate bound
\[
DSO(G)^{2} \le \frac{m^{2}}{2}+\frac{m}{8}\,\sigma(G),
\]
which is an upper bound for $DSO(G)$ solely in terms of $m$ and $\sigma(G)$.
Comparing this with the upper bound of Theorem~\ref{thm:DSO-sigma},
\[
DSO(G) \le \frac{m}{\sqrt{2}} + \frac{\sigma(G)}{8\sqrt{2}\,\delta^{2}},
\]
we observe that neither dominates the other in all cases.

Squaring the latter gives
\[
\left(\frac{m}{\sqrt{2}} + \frac{\sigma(G)}{8\sqrt{2}\,\delta^{2}}\right)^{2}
= \frac{m^{2}}{2} + \frac{m\sigma(G)}{8\delta^{2}} + \frac{\sigma(G)^{2}}{128\,\delta^{4}}.
\]
Hence the difference between the two squared bounds is
\[
\Bigl(\frac{m^{2}}{2}+\frac{m}{8}\sigma(G)\Bigr)
- \Bigl(\frac{m^{2}}{2}+\frac{m\sigma(G)}{8\delta^{2}}+\frac{\sigma(G)^{2}}{128\delta^{4}}\Bigr)
= \frac{m\sigma(G)}{8}\Bigl(1-\frac{1}{\delta^{2}}\Bigr) - \frac{\sigma(G)^{2}}{128\delta^{4}}.
\]

For $\delta=1$ the first term vanishes and the expression equals $-\sigma(G)^{2}/128$, so
\[
\frac{m^{2}}{2}+\frac{m}{8}\sigma(G)
< \Bigl(\frac{m}{\sqrt{2}}+\frac{\sigma(G)}{8\sqrt{2}}\Bigr)^{2},
\]
showing that the bound of Theorem~\ref{alt} is strictly sharper for graphs with minimum degree $1$ (such as stars).
For $\delta\ge 2$, if $\sigma(G)$ is small relative to $m$, the positive term $\frac{m\sigma(G)}{8}(1-\delta^{-2})$ dominates, making the bound of Theorem~\ref{thm:DSO-sigma} tighter.
In summary, the two bounds complement each other: Theorem~\ref{alt} is preferable when $\delta=1$, while Theorem~\ref{thm:DSO-sigma} often gives a better estimate for larger $\delta$ and moderate irregularity.
\end{remark}

Let $G$ be a graph.
Since $Alb(G) \le irr_t(G)$,  Theorem~\ref{new:Alb-upper} implies that
\[
DSO(G) \le \frac{m}{\sqrt{2}} + \frac{irr_t(G)}{4\sqrt{2}\,\delta}.
\]
We now improve this result for non-regular graphs as follows.
\begin{theorem}\label{new:irr-upper-implicit}
Let $G$ be a non-regular graph.
Then
\[
DSO(G) \le \frac{m}{\sqrt{2}} + \frac{1}{4\sqrt{2}\,\delta}
\left( irr_t(G) - \frac{\bigl(n M_1(G)-4m^2\bigr)-\sigma(G)}{\Delta-\delta} \right).
\]
If $G$ is regular, equality $DSO(G)=m/\sqrt{2}$ holds trivially.
\end{theorem}
\begin{proof}
Assume $\Delta>\delta$, since otherwise $G$ is regular and the statement is trivial. For any two numbers $a,b$ with $|a-b|\le \Delta-\delta$,
\[
|a-b| \ge \frac{(a-b)^2}{\Delta-\delta}.
\]
Applying this to every non-edge $\{u,v\}$ and summing yields
\[
\sum_{\{u,v\}\notin E} |d_u-d_v|
\ge \frac{1}{\Delta-\delta}\sum_{\{u,v\}\notin E} (d_u-d_v)^2.
\]
The sum of $(d_u-d_v)^2$ over all unordered vertex pairs equals $n M_1(G)-4m^2$ (a standard identity), while the sum over edges is $\sigma(G)$.
Hence
\[
\sum_{\{u,v\}\notin E} (d_u-d_v)^2 = \bigl(n M_1(G)-4m^2\bigr)-\sigma(G),
\]
and therefore
\begin{equation}\label{eq08}
\sum_{\{u,v\}\notin E} |d_u-d_v|
\ge \frac{\bigl(n M_1(G)-4m^2\bigr)-\sigma(G)}{\Delta-\delta}. 
\end{equation}
Now, $Alb(G)=irr_t(G)-\sum_{\{u,v\}\notin E}|d_u-d_v|$. Using \eqref{eq08} we obtain
\[
Alb(G) \le irr_t(G) - \frac{\bigl(n M_1(G)-4m^2\bigr)-\sigma(G)}{\Delta-\delta}.
\]
By Theorem~\ref{new:Alb-upper},
$DSO(G) \le \frac{m}{\sqrt{2}} + \frac{Alb(G)}{4\sqrt{2}\,\delta}$.
Substituting the upper bound on $Alb(G)$ gives the claimed inequality.
If $G$ is regular, the bound reduces to $m/\sqrt{2}$ which is exact.
\end{proof}
\begin{remark}
The  term
\[
\frac{\bigl(nM_1(G)-4m^2\bigr)-\sigma(G)}
{\Delta-\delta}
=
\frac{\sum_{\{u,v\}\notin E}(d_u-d_v)^2}
{\Delta-\delta}
\]
in Theorem~\ref{new:irr-upper-implicit}
is always nonnegative and is positive whenever there exists a non-edge
whose endpoints have distinct degrees.
Consequently, in such cases the theorem yields a strictly sharper bound than
\[
DSO(G)\le
\frac{m}{\sqrt2}
+\frac{irr_t(G)}{4\sqrt2\,\delta}.
\]
Regular graphs are treated separately in the theorem, where the exact identity
\[
DSO(G)=\frac{m}{\sqrt2}
\]
holds.
\end{remark}

In the last part of this section, we examine the relationship between the Diminished Sombor index and the original Sombor index.

\begin{theorem}\label{strong}
Let $G$ be a nontrivial graph. Then
\[
DSO(G)\ge \frac{SO(G)^2}{\sqrt{F(G)\,HM(G)}}.
\]
Equality holds if and only if $G$ is either a regular graph or a bipartite semiregular graph.
\end{theorem}
\begin{proof}
For each edge $e=uv\in E$, define $a_e = \sqrt{d_u^2+d_v^2}$ and $c_e = 1/(d_u+d_v)$. Apply the Engel form of the Cauchy--Schwarz inequality:
\[
\Bigl(\sum_{e\in E} a_e\Bigr)^2 \le \Bigl(\sum_{e\in E} \frac{a_e}{c_e}\Bigr) \Bigl(\sum_{e\in E} a_e c_e\Bigr).
\]
Since $a_e/c_e = (d_u+d_v)\sqrt{d_u^2+d_v^2}$ and $a_e c_e = \frac{\sqrt{d_u^2+d_v^2}}{d_u+d_v}$, we obtain
\[
(SO(G))^2 \le \Bigl(\sum_{uv\in E(G)} (d_u+d_v)\sqrt{d_u^2+d_v^2}\Bigr) DSO(G).
\]
Now apply the standard Cauchy--Schwarz inequality to the sum on the right:
\[
\Bigl(\sum_{uv\in E(G)} (d_u+d_v)\sqrt{d_u^2+d_v^2}\Bigr)^2
\le \Bigl(\sum_{uv\in E(G)} (d_u+d_v)^2\Bigr) \Bigl(\sum_{uv\in E(G)} (d_u^2+d_v^2)\Bigr)
\]
\[
= HM(G)\,F(G).
\]
All terms nonnegative, hence taking square roots gives
\[
\sum_{uv\in E(G)} (d_u+d_v)\sqrt{d_u^2+d_v^2} \le \sqrt{HM(G)\,F(G)}.
\]
Therefore,
\[
(SO(G))^2 \le \sqrt{HM(G)\,F(G)}\; DSO(G),
\]
which rearranges to the required bound.

Equality in the Engel inequality holds if and only if
$\sqrt{\frac{a_e}{c_e}}$ is proportional to $\sqrt{a_ec_e}$ if and only if $\frac{1}{c_e}$ is constant.
This is equivalent to saying that
$d_u+d_v$ is constant over all edges. 
Equality in the second Cauchy--Schwarz inequality holds if and only if
$\frac{d_u^2+d_v^2}{(d_u+d_v)^2}$ is
constant over all edges.
Combining these two facts shows that equality holds if and only if both
$d_u+d_v$
and
$d_u^2+d_v^2$
are constant on $E(G)$.
In particular,  by Lemma~\ref{const} the graph is either regular or bipartite semiregular.
Conversely, if $G$ is regular or bipartite semiregular, then both
$d_u+d_v$ and $d_u^2+d_v^2$ are constant on all edges.
Hence equality holds in the Engel inequality 
and in the second Cauchy--Schwarz inequality.
Thus all intermediate inequalities become equalities, and the bound is attained.
\end{proof}
\begin{corollary}\label{weak}
For every nontrivial graph $G$,
\[
DSO(G)\ge \frac{SO(G)^2}{\sqrt{2}\,F(G)}.
\]
Equality holds if and only if $G$ is regular.
\end{corollary}
\begin{proof}
Since $(d_u+d_v)^2\le 2(d_u^2+d_v^2)$ for every edge $uv$, we have $HM(G)\le 2F(G)$, hence
\[
\sqrt{F(G)\,HM(G)}\le \sqrt{2}\,F(G).
\]
Therefore, from Theorem~\ref{strong},
\[
DSO(G)\ge \frac{SO(G)^2}{\sqrt{F(G)\,HM(G)}}
\ge \frac{SO(G)^2}{\sqrt{2}\,F(G)}.
\]
The equality conditions follow directly from those of Theorem~\ref{strong}, noting that equality in $HM(G)\le 2F(G)$ forces $d_u=d_v$ for every edge, i.e., $G$ is regular.
\end{proof}

\begin{remark}\label{remn}
It is worth emphasizing that even the weaker bound of Corollary~\ref{weak} improves the previously known lower bound
\[
DSO(G)\ge \frac{SO(G)}{2\Delta},
\]
which was established in Theorem 3.11 of \cite{5}. 
Indeed, since $d_u^2+d_v^2\le 2\Delta^2$ for each edge, we have
$
\sqrt{d_u^2+d_v^2}\ge \frac{d_u^2+d_v^2}{\sqrt{2}\,\Delta}.
$
Summing over all edges yields $SO(G)\ge \frac{F(G)}{\sqrt{2}\,\Delta}$. Consequently,
\[
\frac{SO(G)^2}{\sqrt{2}\,F(G)}
\ge \frac{SO(G)}{2\Delta}.
\]
Thus the bound of Corollary~\ref{weak}
dominates the bound from \cite{5}, with equality if and only if $G$ is regular. 
\end{remark}

\begin{proposition}\label{DSO-SO-upper}
Let $G$ be a graph with $m>0$ edges and minimum degree $\delta\ge 1$. Then
\[
DSO(G) \le \frac{SO(G)}{2\delta}.
\]
Equality holds if and only if $G$ is regular.
\end{proposition}
\begin{proof}
For every edge $uv$, $d_u+d_v \ge 2\delta$ implies $\frac{1}{d_u+d_v} \le \frac{1}{2\delta}$. Hence
\[
\frac{\sqrt{d_u^2+d_v^2}}{d_u+d_v} \le \frac{\sqrt{d_u^2+d_v^2}}{2\delta},
\]
and summing yields the bound. Equality forces $d_u+d_v = 2\delta$ for all edges, which together with $d_u,d_v\ge \delta$ gives $d_u=d_v=\delta$; thus $G$ is $\delta$-regular. Conversely, if $G$ is regular of degree $r$, then $\delta=r$,
$SO(G)=m\sqrt{2r^2}=mr\sqrt{2}$, and $DSO(G)=m/\sqrt{2}$.
The right‑hand side equals $(mr\sqrt{2})/(2r)=m/\sqrt{2}$, so equality holds.
\end{proof}

\begin{remark}\label{rem:SO-F-HM-chain}
By the Cauchy--Schwarz inequality,
\[
SO(G)^2 = \Bigl(\sum_{uv\in E(G)} 1\cdot\sqrt{d_u^2+d_v^2}\Bigr)^2
        \le m\sum_{uv\in E(G)}(d_u^2+d_v^2) = m\,F(G),
\]
with equality if and only if $\sqrt{d_u^2+d_v^2}$ is constant over all edges.
Lemma~\ref{const} implies  that $G$ is either regular or bipartite semiregular.
Combining this with Proposition~\ref{DSO-SO-upper} yields
\[
DSO(G) \le \frac{SO(G)}{2\delta} \le \frac{\sqrt{m\,F(G)}}{2\delta}.
\]
The right‑hand side is precisely the upper bound of Theorem~\ref{DSO-F}.  
Furthermore, because every edge has $d_u+d_v \ge 2$, we have $HM(G) = \sum (d_u+d_v)^2 \ge 4m$ (equality if and only if $G\cong K_2$). Thus
\[
\frac{\sqrt{m\,F(G)}}{2\delta} \le \frac{\sqrt{F(G)\,HM(G)}}{4\delta}.
\]
Putting everything together we obtain the chain
\[
DSO(G) \le \frac{SO(G)}{2\delta}
      \le \frac{\sqrt{m\,F(G)}}{2\delta}
      \le \frac{\sqrt{F(G)\,HM(G)}}{4\delta},
\]
where the first inequality is sharp exactly for regular graphs, the second when
$G$ is either regular or bipartite semiregular, and the third when $G\simeq K_2$.
Thus the upper bound of Theorem~\ref{DSO-F} appears as a natural intermediate step linking the Sombor index, the forgotten index, and the hyper‑Zagreb index.
\end{remark}

\begin{remark}
Together with Theorem~\ref{strong}, we now have the two‑sided estimate
\[
\frac{SO(G)^2}{\sqrt{F(G)\,HM(G)}} \le DSO(G) \le \frac{SO(G)}{2\delta},
\]
where both bounds are sharp for regular graphs. The upper bound can be weakened to involve the forgotten and hyper‑Zagreb indices alone, as shown in the chain above, at the cost of losing sharpness for regular graphs of degree larger than $1$.
\end{remark}


\begin{thebibliography}{00}

\bibitem{25}
H. Abdo, S. Brandt, D. Dimitrov, The total irregularity of a graph, \textit{Discrete Math. Theor. Comput. Sci.}, \textbf{16} (2014), 201–206.

\bibitem{20}
M. O. Albertson, The irregularity of graph, \textit{Ars Comb.}, \textbf{46} (1997), 219–225.

\bibitem{33}
A. M. Alotaibi, A. M. Alanazi, T. S. Hassan, A. Ali, On the diminished Sombor index of fixed-order molecular graphs with cyclomatic number at least 3, \textit{MATCH Commun. Math. Comput. Chem.}, \textbf{95} (2026), 813–828.

\bibitem{7}
J. A. Bondy, U. S. R. Murty, \textit{Graph Theory}, Springer, London (2008).

\bibitem{4}
A. E. Brouwer, W. H. Haemers, Spectra of graphs, \textit{Springer Science and Business Media}, (2011).

\bibitem{8}
G. Chartrand, L. Lesniak, P. Zhang, Graphs and Digraphs, Sixth Edition, \textit{CRC Press, Boca Raton} (2016).

\bibitem{9}
R. Diestel, Graph Theory, Third Edition, \textit{Springer, New York} (2005).

\bibitem{14}
T. Došlić, T. Réti, A. Ali, On the structure of graphs with integer Sombor indices, \textit{Discrete Math. Lett.}, \textbf{7} (2021), 1–4.

\bibitem{22}
S. Fajtlowicz, On conjectures of Graffiti-II, \textit{Congr. Numer.}, \textbf{60} (1987), 187–197.

\bibitem{21}
G. H. Fath-Tabar, Old and new Zagreb indices of graphs, \textit{MATCH Commun. Math. Comput. Chem.}, \textbf{65}(1) (2011), 79–84.

\bibitem{24}
B. Furtula, I. Gutman, A forgotten topological index, \textit{J. Math. Chem.}, \textbf{53} (2015), 1184–1190.

\bibitem{31}
F. Guo, F. Wang, Maximum diminished Sombor index of molecular trees with a perfect matching, \textit{arXiv:2512.19710} (2025).

\bibitem{10}
I. Gutman, B. Furtula, C. Elphick, Three new/old vertex–degree–based topological indices, \textit{MATCH Commun. Math. Comput. Chem.}, \textbf{72} (2014), 617–682.

\bibitem{29}
I. Gutman, M. Togan, A. Yurttas, A. S. Cevik, I. N. Cangul, Inverse problem for sigma index, \textit{MATCH Commun. Math. Comput. Chem.}, \textbf{79} (2018), 491–508.


\bibitem{28}
I. Gutman, Geometric approach to degree-based topological indices: Sombor indices, \textit{MATCH Commun. Math. Comput. Chem.}, \textbf{86}(1) (2021), 11–16.

\bibitem{15}
I. Gutman, N. K. Gürsoy, A. Gürsoy, A. Ülker, New bounds on Sombor index, \textit{Commun. Comb. Optim.}, \textbf{8} (2023), 305–311.

\bibitem{11}
B. Liu, Z. You, A survey on comparing Zagreb indices, \textit{MATCH Commun. Math. Comput. Chem.}, \textbf{65} (2011), 581–593.

\bibitem{16}
I. Milovanović, E. Milovanović, A. Ali, M. Matejić, Some results on the Sombor indices of graphs, \textit{Contrib. Math.}, \textbf{3} (2021), 59–67.

\bibitem{5}
F. Movahedi, Diminished Sombor index and its relationship with topological indices, \textit{Filomat}, \textbf{39}(29) (2025), 10519–10532.

\bibitem{3}
F. Movahedi, I. Gutman, I. Redžepović, B. Furtula, Diminished Sombor index, \textit{MATCH Commun. Math. Comput. Chem.}, \textbf{95} (2026), 141–162.

\bibitem{30}
F. Movahedi, Maximal Diminished Sombor index of tricyclic graphs, \textit{MATCH Commun. Math. Comput. Chem.}, \textbf{95} (2026), 935--939.

\bibitem{17}
C. Phanjoubam, S. M. Mawiong, A. M. Buhphang, On Sombor coindex of graphs, \textit{Commun. Comb. Optim.}, \textbf{8} (2023), 513–529.

\bibitem{13}
D. T. Rajathagiri, Enhanced mathematical models for the Sombor index: Reduced and co-Sombor index perspectives, \textit{Data Anal. Artif. Intell.}, \textbf{1}(2) (2021), 215–228.

\bibitem{18}
M. Randić, Characterization of molecular branching, \textit{J. Am. Chem. Soc.}, \textbf{97}(23) (1975), 6609–6615.

\bibitem{26}
G. H. Shirdel, H. Rezapour, A. M. Sayadi, The Hyper-Zagreb index of graph operations, \textit{Iran. J. Math. Chem.}, \textbf{4}(2) (2013), 213–220.

\bibitem{19}
D. Vukičević, B. Furtula, Topological index based on the ratios of geometrical and arithmetical means of end-vertex degrees of edges, \textit{J. Math. Chem.}, \textbf{46} (2009), 1369–1376.

\bibitem{23}
D. Vukičević, M. Gašperov, Bond additive modeling 1. Adriatic indices, \textit{Croat. Chem. Acta.}, \textbf{83} (2010), 243–260.


\end{thebibliography}
\end{document}